\documentclass{article}
\usepackage{maa-monthly}
\usepackage{amsfonts,amssymb}
\usepackage{graphicx}
\usepackage[margin=1in]{geometry}
\usepackage{amsthm}
\usepackage{mathtools}
\usepackage{tikz}
\usepackage[T1]{fontenc}
\usepackage{textcomp}
\usepackage{diagbox}

\usepackage{colortbl}

\usepackage{multirow}
 \usepackage{xcolor}
\usepackage{color}
\usepackage{float}
\usepackage{soul}
\definecolor{labelkey}{rgb}{0,0.08,0.45}
\definecolor{refkey}{rgb}{0,0.6,0.0}

\definecolor{Brown}{rgb}{0.45,0.0,0.05}
\definecolor{lime}{rgb}{0.00,0.8,0.0}
\definecolor{lblue}{rgb}{0.5,0.5,0.99}
\definecolor{lblue}{rgb}{0.8,0.85,1.00}
\definecolor{anotherblue}{rgb}{.8, .8,1}
\definecolor{violet}{rgb}{0.9,0.6,0.9}
\definecolor{greenyellow}{rgb}{0.53,0.99,0.18}

\definecolor{Lyellow}{rgb}{0.87,0.87,0.87}
\definecolor{Lgray}{rgb}{0.92,0.92,0.92}
\definecolor{Mgray}{rgb}{0.5,0.5,0.5}
\definecolor{Gold}{rgb}{0.99,0.84,0.0}

\usepackage{siunitx}

\graphicspath{{./figures/}}
\usepackage[colorinlistoftodos,prependcaption,textsize=tiny]{todonotes}
\usepackage[shortlabels]{enumitem}
\usepackage{mathrsfs}

\usepackage{cite}
\usepackage{color}
\usepackage{graphicx}
\usepackage{caption}
\usepackage{subcaption}

\usepackage[utf8]{inputenc}

\usepackage{pstricks}

\theoremstyle{theorem}
\newtheorem{theorem}{Theorem}

\newtheorem{lemma}[theorem]{Lemma}

\theoremstyle{definition}
\newtheorem*{definition}{Definition}
\newtheorem*{remark}{Remark}

\begin{document}
 
\title{On the Recurrence Formula for Fixed Points of the Josephus Function}
\markright{On the Recurrence Formula for Fixed Points}
\author{Yunier Bello-Cruz and Roy Quintero-Contreras}

\maketitle
\begin{abstract}
In this paper, we provide a comprehensive solution to the open problem regarding the existence of a recurrence formula for computing fixed points of the Josephus function precisely when the reduction constant is three. Incorporating this formula into recursive algorithms significantly improves addressing the Josephus problem, particularly for large inputs.
\end{abstract}

\section{Introduction} \label{sec:intro} 

The Josephus problem \cite{Jos,Knuth,Tai} presents a combinatorial challenge involving the arrangement of \(n\) participants in a circular formation, followed by a subsequent elimination process. Specifically: a starting point is chosen, a direction of rotation is determined, and participants are enumerated around the circle until \(k-1\) have been counted. Here, \(k\) serves as the reduction constant. The \(k^{th}\) participant is then eliminated and removed from the circle. This procedure continues until only one remains. The primary aim of this problem is to ascertain the initial position of this survivor, denoted \(J_{_k}(n)\). In this paper, we focus on the case where \(k=3\) and address an open question raised in \cite{Bel} regarding the existence of a recurrence formula for sequentially calculating fixed points of the Josephus function.

The Josephus problem is a classic example of a recurrence relation, which has been extensively studied in the literature. One of the earliest formal methodologies for solving the original problem, where $k=3$ and $n=41$, was presented by Bachet \cite{Bac}. Subsequently, Euler \cite{Eul} studied a recursive relation on $n$ for the general case, establishing a connection between the survivor's position $J_{_k}(n)$ and its preceding position $J_{_k}(n-1)$, marking a noteworthy advancement in the mathematical approach for understanding the problem. The Josephus problem has captivated mathematicians across diverse domains, triggering expansions in the examination of permutations and contemporary applications in computer algorithms, data structures, and image encryption; see, for instance, \cite{Cha, Nai, Yan, Wil, Hal, Her, Asv}. When the reduction constant is $2$, the problem has been entirely solved by Knuth \cite{Gra}, who deduced a closed-form expression for $J_{_2}(n)$. Additionally, Knuth formulated an efficient algorithm for evaluating the general Josephus function $J_{_k}$ and inspired a series of results that unleashed the recursive nature intrinsic to the problem; for specific examples, consult \cite{Odl, Rob, Tai, Dow, Woo, Cos, Uch, Chuang} and the references therein.

A recent study \cite{Bel} introduced a non-recursive strategy based on recurrence formulas between extremal points, drawing upon the discrete piecewise linear structure inherent to the Josephus function. This novel approach offered an efficient solution to the Josephus problem and led to a conjecture about the potential to enhance the proposed algorithm using only fixed points, which are special extremal points. This paper seeks to address this open question from \cite{Bel} by proposing an explicit recurrence formula for the fixed points of the classical Josephus function \(J_{_3}\). Possessing a recurrence formula for these fixed points is of significant importance as it paves the way for improving recursive algorithms that solve the Josephus problem.

This paper is organized as follows: The remainder of this section will provide the necessary background and notation. Section \ref{sec:fixedpoints} presents the main results of this paper, including the recurrence formula for the fixed points and a closed expression to evaluate the Josephus function $J_{_3}$. These formulas can be used to improve the extremal algorithm for solving the Josephus problem. 
Section \ref{sec:conclusion} includes concluding remarks of the paper where a brief discussion of the results and future research directions are presented.

\subsection{Notation and Definitions} \label{sec:preliminaries}
The mathematical formulation of the classical Josephus problem can be formulated as follows: Let $n$ people be arranged in a circle that closes up its ranks as individuals are picked out. Starting anywhere (person $1^{\rm st}$ spot), go sequentially around clockwise, picking out each third person (the \textit{reduction constant}) until one person is left (the \textit{survivor}). The position of the survivor is denoted by $J_{_3}(n)$, which belongs to the natural numbers $\mathbb{N}$. This procedure is called the \textit{elimination process}, and it naturally generates a discrete function $J_{_3}: \mathbb{N} \rightarrow \mathbb{N}$ that we will call the Josephus function. We say that the Josephus problem has been solved once we have determined the value of $J_{_3}$ at $n$.

We denote by $[[\ell, m]]$, the set $\{\ell,\ldots,m\}$ for any two integers $\ell$ and $m$ such that $\ell\le m$. Note further that $J_{_3}(n) \in [[1,n]]$ for every $n$. 

The discrete piecewise linear structure of the Josephus function naturally leads to the following definition of extremal points of $J_{_3}$.

\begin{definition}[Extremal and Fixed points]
A \textit{high extremal point} $n_{_e}$ is defined as a point that satisfies $J_{_3}(n_{_e}) \in \{n_{_e}-1,n_{_e}\}$. Especially, we say that a high extremal point $n_{_p}$  is a \textit{fixed point} of $J_{_3}$ if $J_{_3}(n_{_p})=n_{_p}$. On the other hand,  if  for $\check n_{_e}$ holds that $J_{_3}(\check n_{_e})\in \{1,2\}$, we refer to $\check n_{_e}$ as a \textit{low extremal point}. 
 \end{definition}

Note that a fixed point $n_{_p}$ is also a high extremal point. However, there are high extremal points that are not fixed points, which will be called \emph{pure} high extremal points. Moreover, for $n\geq 3$, a sequence of distinct high and low extremal points exists, and the Josephus function exhibits a piecewise linear structure between these extremal points; see Figure \ref{f1} for a better illustration of these features. 

\begin{figure}[h]
\centering
\includegraphics[width=0.81\textwidth]{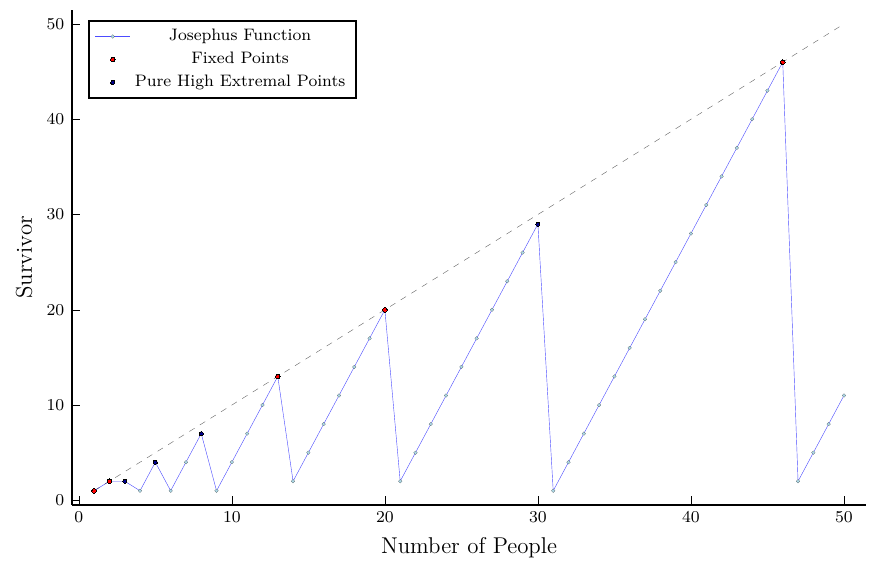}
\hfill
\caption{Graph of the Josephus function $J_{_3}$ for $n\le 50$.}
\label{f1}
\end{figure}

Denote $\{n_{_e}^{(i)}\}_{i\in \mathbb{N}}$ as the sequence of high extremal points for $J_{_3}$ increasingly distributed starting at $n_{_e}^{(1)}=1$. We also introduce the fixed point indicator, $f_i$, as the binary sequence defined by $f_i:=n_{_e}^{(i)}-J_{_3}(n_{_e}^{(i)})$ for all $i$. Note that $f_i$ describes whether the high extremal point $n_{_e}^{(i)}$ is a fixed point or not, \emph{i.e.}, $f_i=0$ if and only if $n_{_e}^{(i)}$ is a fixed point, or equivalent, $f_i=1$ if and only if $n_{_e}^{(i)}$ is a pure high extremal point.

\section{Properties of high extremal points and recurrence formulas} \label{sec:fixedpoints}

Next we recall several relevant properties of the Josephus function. We start by an important result involving high extremal points.

\begin{theorem}[Four Cases of High Extremal Points]\label{Thm-cases}
Let $\{n_{_e}^{(i)}\}_{i\in \mathbb{N}}$ be the sequence of high extremal points for $J_{_3}$ such that $n_{_e}^{(1)}=1$ and define $f_i:=n_{_e}^{(i)}-J_{_3}(n_{_e}^{(i)})$ and $r_i:={\rm mod}({n}_{_e}^{(i)},2)$ for all $i$.
    Then, the following statements hold:

\item[ {\bf (i)}] If $f_{i}=1$ and $r_i=0$ then  ${n}_{_e}^{(i+1)}= \dfrac{3{n}_{_e}^{(i)}+2}{2}$ and $f_{i+1}=0$.

\item[ {\bf (ii)}] If $f_{i}=1$ and $r_i=1$ then  ${n}_{_e}^{(i+1)}= \dfrac{3{n}_{_e}^{(i)}+1}{2}$ and $f_{i+1}=1$.

\item[ {\bf (iii)}] If $f_{i}=0$ and $r_i=0$ then  ${n}_{_e}^{(i+1)}= \dfrac{3{n}_{_e}^{(i)}}{2} $ and $f_{i+1}=1$.

\item[ {\bf (iv)}] If $f_{i}=0$ and $r_i=1$ then  ${n}_{_e}^{(i+1)}= \dfrac{3{n}_{_e}^{(i)}+1}{2}  $ and $f_{i+1}=0$.
\end{theorem}

\begin{proof}
    
To prove (i), note that if $f_i=1$ and $r_i:={\rm mod}({n}_{_e}^{(i)},2)=0$ (or ${n}_{_e}^{(i)}$ even), Equations (11), (12) and (13) in Corollary $7$ of \cite{Bel} can be used to show that $c_i=0$ and hence,
\begin{align*}
{n}_{_e}^{(i+1)}&=\dfrac{3({n}_{_e}^{(i)}+1)-(0+1)}{2}-0 =\dfrac{3{n}_{_e}^{(i)}+2}{2}.
\end{align*} Note further that Equation (14) in Corollary $7$ of \cite{Bel} gives us
$$
J_{_3}(n_{_e}^{(i+1)})=2-f_i+3\left\lfloor\dfrac{{n}_{_e}^{(i)}+f_i-1}{2}\right\rfloor=1+3\left\lfloor\dfrac{{n}_{_e}^{(i)}}{2}\right\rfloor=\dfrac{3{n}_{_e}^{(i)}+2}{2},
$$ 
using the fact that ${n}_{_e}^{(i)}$ is even in the last equality. So, $f_{i+1}=n_{_e}^{(i+1)}-J_{_3}(n_{_e}^{(i+1)})=0$, which completes the proof of item (i).

The proofs for remainder cases {\bf (ii)}, {\bf (iii)}, and {\bf (iv)} follow a similar method, using Corollary $7$ from \cite{Bel}.
\end{proof}

Note that the above result gives us a clear picture of the behavior of the sequence of high extremal points for $J_{_3}$, $\{n_{_e}^{(i)}\}_{i\in \mathbb{N}}$. In particular, if $n_{_e}^{(i)}$ is a fixed point, $n_{_e}^{(i+1)}$ is a fixed point if and only if $n_{_e}^{(i)}$ is odd. On the other hand, if $n_{_e}^{(i)}$ is not a fixed point, $n_{_e}^{(i+1)}$ is a fixed point if and only if $n_{_e}^{(i)}$ is even. This important observation will be used below to understand how many pure high extremal points are between two consecutive fixed points, which is crucial to derive a recurrence formula for computing consecutive fixed points.  

Next, we will unify the four cases of the above theorem, which allows us to establish simplified recurrence formulas between consecutive high extremal points, $n_{_e}^{(i)}$, their fixed point indicators, $f_i$, their parities, $r_i$, and the expression for their functional values, $J_{_3}({n}_{_e}^{(i)})$.

\begin{lemma}[Recurrence Formulas for High Extremal Points] Let $\{n_{_e}^{(i)}\}_{i\in \mathbb{N}}$ be the sequence of high extremal points for $J_{_3}$ starting at $n_{_e}^{(1)}=1$ and define $r_i:={\rm mod} (n_{_e}^{(i)}, 2)$, $f_i:=n_{_e}^{(i)}-J_{_3}(n_{_e}^{(i)})$, and $s_i:={\rm mod} \left ((3n_{_e}^{(i)}+2-r_{_i})/2, 2\right )$ for all $i$. Then, 
\begin{align}\label{heprecurrence}
{n}_{_e}^{(i+1)}&=\dfrac{3{n}_{_e}^{(i)}+1+(1-r_{_i})(2f_{i}-1)}{2}, \\\label{heprecurrenceJ}
J_{_3}({n}_{_e}^{(i+1)})&=\dfrac{3{n}_{_e}^{(i)}+(2-3r_{_i})(2f_{i}-1)}{2},\\\label{heprecurrencef}
f_{i+1}&=f_i-(1-r_i)(2f_i-1), \\\label{heprecurrencer}
r_{i+1}&=s_{_i}-(1-r_{_i})(1-f_i)(2s_{_i}-1).
\end{align}
\end{lemma}

\begin{proof} The formulas for ${n}_{_e}^{(i+1)}$ given in the items (i) to (iv) of Theorem \ref{Thm-cases} coincide with \eqref{heprecurrence} after noting that
    $$\left. \begin{cases}
     f_{i}=1, r_{_i}=0 \\
    f_{i}=1, r_{_i}=1 \\
    f_{i}=0, r_{_i}=0 \\
    f_{i}=0, r_{_i}=1        
    \end{cases}\hspace{-.45cm} \right . \quad {\rm imply} \quad 1+(1-r_{_i})(2f_{i}-1)= \left. \begin{cases}
    2 \\
    1 \\
    0 \\
    1 ,
    \end{cases} \right. \; {\rm and}\; {\rm then} \quad \left. \begin{cases}
      {n}_{_e}^{(i+1)}= \dfrac{3{n}_{_e}^{(i)}+2}{2}  \\
      {n}_{_e}^{(i+1)}= \dfrac{3{n}_{_e}^{(i)}+1}{2} \\
      {n}_{_e}^{(i+1)}= \dfrac{3{n}_{_e}^{(i)}}{2} \\
      {n}_{_e}^{(i+1)}= \dfrac{3{n}_{_e}^{(i)}+1}{2}.
      \end{cases}\hspace{-.45cm} \right.$$
In similar way, note that if
$$\left. \begin{cases}
f_{i}=1, r_{_i}=0 \\
    f_{i}=1, r_{_i}=1 \\
    f_{i}=0, r_{_i}=0 \\
    f_{i}=0, r_{_i}=1        
    \end{cases}\hspace{-.45cm} \right . \quad {\rm then} \quad f_i-(1-r_i)(2f_i-1)= f_{i+1}=\left.\begin{cases}
        0  \\
         1 \\
         1 \\
          0 ,  
         \end{cases}\hspace{-.45cm} \right.$$
which proves \eqref{heprecurrencef}. On the other hand, the functional value expression of ${n}_{_e}^{(i+1)}$, $J_{_3}({n}_{_e}^{(i+1)})$, in \eqref{heprecurrenceJ} follows directly from the fact that $J_{_3}(n_{_e}^{(i+1)})={n}_{_e}^{(i+1)}-f_{i+1}$, which can be rewritten by using \eqref{heprecurrence} and \eqref{heprecurrencef} as follows:
\begin{align*}
    J_{_3}(n_{_e}^{(i+1)})=&\dfrac{3{n}_{_e}^{(i)}+1-2f_{i+1}+(1-r_{_i})(2f_{i}-1)}{2}\\=&\dfrac{3{n}_{_e}^{(i)}+1-2f_i+3(1-r_{_i})(2f_{i}-1)}{2}\\=&\dfrac{3{n}_{_e}^{(i)}+(2-3r_{_i})(2f_{i}-1)}{2}.
\end{align*}

To prove \eqref{heprecurrencer}, we recursively generate the sequence $\{r_{_i}\}_{i\in \mathbb{N}}$, starting at $r_{_1}={\rm mod}(n_{_e}^{(1)})={\rm mod}(1)=1$ and defining $s_i:={\rm mod} \left (\dfrac{3n_{_e}^{(i)}+2-r_{_i}}{2}, 2\right )$ for all $i$, which is well-defined because $\dfrac{3n_{_e}^{(i)}+2-r_{_i}}{2}$ is always an integer number. Note that \eqref{heprecurrence} implies that 
$$
{n}_{_e}^{(i+1)}=\dfrac{3n_{_e}^{(i)}+2-r_i+(1-r_i)(2f_i-2)}{2}= \dfrac{3n_{_e}^{(i)}+2-r_i}{2}-(1-r_i)(1-f_i).$$
Hence, $r_{i+1}=s_i-(1-r_i)(1-f_i)(2s_i-1)$. One may verify that the sequence $\{r_{_i}\}_{i\in \mathbb{N}}$ is generated by the following rule:
\begin{align*}
  \text{ if }  \hspace{.5cm} \begin{dcases}
    f_{_i}=1, r_{_i}=1, s_{_i}=0 \\
    f_{_i}=0, r_{_i}=1, s_{_i}=0\\
    f_{_i}=1, r_{_i}=0, s_{_i}=0 \\
    f_{_i}=0, r_{_i}=0, s_{_i}=1 ,   
\end{dcases} \;\text{ then }\;  \hspace{.5cm} s_{_i}-(1-r_{_i})(1-f_{_i})(2s_{_i}-1)=r_{i+1}=0, 
\end{align*}and 
\begin{align*} 
 \text{ if } \hspace{.5cm} \begin{dcases}
  f_{_i}=1, r_{_i}=0, s_{_i}=1 \\
  f_{_i}=0, r_{_i}=0, s_{_i}=0 \\
  f_{_i}=1, r_{_i}=1, s_{_i}=1 \\
   f_{_i}=0, r_{_i}=1, s_{_i}=1,   
\end{dcases} \;\text{ then }\;  \hspace{.5cm} s_{_i}-(1-r_{_i})(1-f_{_i})(2s_{_i}-1)=r_{i+1}=1.
\end{align*}
Hence, the recursive formula \eqref{heprecurrencer} holds.

\end{proof}
We denote the high extremal point $n_{_e}$ as $n_{_p}$ whenever $n_{_e}$ is a fixed point. Let us now focus on investigating the sequence of fixed points $\{n_{_p}^{(\ell)}\}_{\ell\in\mathbb{N}}$ of the Josephus function $J_{_3}$. It is crucial to remember that while every fixed point is automatically a high extremal point, the converse is not always true. Additionally, we should keep in mind that there is an infinite number of fixed points for $J_{_3}$, as stated in Theorem $9$ of \cite{Bel}. These fixed points form a sequence that starts in $1$ and always increases. The behavior of the sequence of fixed points is intriguing and complex as well as the number of pure high extremal points between consecutive fixed points $n_{_p}^{(\ell)}$ and $n_{_p}^{(\ell+1)}$, which will be denoted by $\overline{m}_{\ell}$; see the following table.

\begin{table}[h]
    \centering
    \begin{tabular}{|c|r|c||c|r|c||c|r|c|}
    \hline
    \multicolumn{1}{|c|}{\cellcolor[gray]{0.89} $\ell$} & \multicolumn{1}{|c|}{\cellcolor[gray]{0.89} $n_{_p}^{(\ell)}$} & \multicolumn{1}{|c||}{\cellcolor[gray]{0.89} $\overline{m}_{\ell}$} & \multicolumn{1}{|c|}{\cellcolor[gray]{0.89}$\ell$} & \multicolumn{1}{|c|}{\cellcolor[gray]{0.89}$n_{_p}^{(\ell)}$} & \multicolumn{1}{|c|}{\cellcolor[gray]{0.89}$\overline{m}_{\ell}$} & \multicolumn{1}{|c|}{\cellcolor[gray]{0.89}$\ell$} & \multicolumn{1}{|c|}{\cellcolor[gray]{0.89}$n_{_p}^{(\ell)}$} & \multicolumn{1}{|c|}{\cellcolor[gray]{0.89}$\overline{m}_{\ell}$} \\ 
    \hline \hline
    \cellcolor[gray]{0.89}1 & \num{1} & 0 & \cellcolor[gray]{0.89}14 & \num{103690} & 5 & \cellcolor[gray]{0.89}27 & \num{29824201117} & 0 \\ \hline
    \cellcolor[gray]{0.89}2 & \num{2} & 3 & \cellcolor[gray]{0.89}15 & \num{1181101} & 0 & \cellcolor[gray]{0.89}28 & \num{44736301676} & 1 \\ \hline
    \cellcolor[gray]{0.89}3 & \num{13} & 0 & \cellcolor[gray]{0.89}16 & \num{1771652} & 1 & \cellcolor[gray]{0.89}29 & \num{100656678772} & 1 \\ \hline
    \cellcolor[gray]{0.89}4 & \num{20} & 1 & \cellcolor[gray]{0.89}17 & \num{3986218} & 7 & \cellcolor[gray]{0.89}30 & \num{226477527238} & 2 \\ \hline
    \cellcolor[gray]{0.89}5 & \num{46} & 2 & \cellcolor[gray]{0.89}18 & \num{102162424} & 1 & \cellcolor[gray]{0.89}31 & \num{764361654430} & 2 \\ \hline
    \cellcolor[gray]{0.89}6 & \num{157} & 0 &\cellcolor[gray]{0.89} 19 & \num{229865455} & 0 & \cellcolor[gray]{0.89}32 & \num{2579720583703} & 0 \\ \hline
    \cellcolor[gray]{0.89}7 & \num{236} & 1 & \cellcolor[gray]{0.89}20 & \num{344798183} & 0 & \cellcolor[gray]{0.89}33 & \num{3869580875555} & 0 \\ \hline
    \cellcolor[gray]{0.89}8 & \num{532} & 1 & \cellcolor[gray]{0.89}21 & \num{517197275} & 0 & \cellcolor[gray]{0.89}34 & \num{5804371313333} & 0 \\ \hline
    \cellcolor[gray]{0.89}9 & \num{1198} & 2 & \cellcolor[gray]{0.89}22 & \num{775795913} & 0 & \cellcolor[gray]{0.89}35 & \num{8706556970000} & 1 \\ \hline
    \cellcolor[gray]{0.89}10 & \num{4045} & 0 & \cellcolor[gray]{0.89}23 & \num{1163693870} & 2 & \cellcolor[gray]{0.89}36 & \num{19589753182501} & 0 \\ \hline
    \cellcolor[gray]{0.89}11 & \num{6068} & 1 & \cellcolor[gray]{0.89}24 & \num{3927466813} & 0 & \cellcolor[gray]{0.89}37 & \num{29384629773752} & 1 \\ \hline
    \cellcolor[gray]{0.89}12 & \num{13654} & 2 & \cellcolor[gray]{0.89}25 & \num{5891200220} & 1 & \cellcolor[gray]{0.89}38 & \num{66115416990943} & 0 \\ \hline
    \cellcolor[gray]{0.89}13 & \num{46084} & 1 & \cellcolor[gray]{0.89}26 & \num{13255200496}  & 1 & \cellcolor[gray]{0.89}39 & \num{99173125486415} & 0 \\ \hline 
    \end{tabular}
    \caption{Values of $n_{_p}^{(\ell)}$ and $\overline{m}_{\ell}$ for $\ell\in\{1,2,\ldots,39\}$.}
    \label{fixed-points-values}
\end{table}

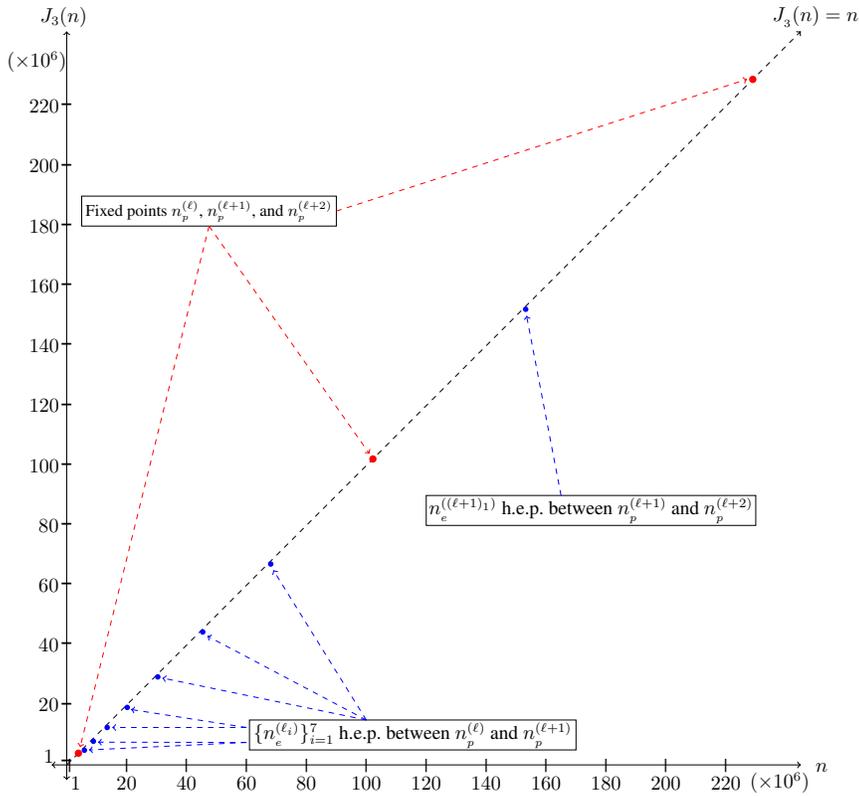
\begin{figure}
    \centering
    \resizebox{0.7\textwidth}{!}{
\begin{tikzpicture}[scale=0.0575]

    \draw (5,180) rectangle (90,190) node[pos=.5] {\small{Fixed points $n_{_p}^{(\ell)}$, $n_{_p}^{(\ell+1)}$, and $n_{_p}^{(\ell+2)}$}};
    
    \draw[dashed,red,->] (47.5,180) -- (4.5,6);
    
    \draw[dashed,red,->] (47.5,180) -- (101,104);
    
    \draw[dashed,red,->] (90,185) -- (227,229);
    
    \draw (61,5) rectangle (170,15) node[pos=.5] {$\{n_{_e}^{(\ell_i)}\}_{i=1}^{7}$ h.e.p. between $n_{_p}^{(\ell)}$ and $n_{_p}^{(\ell +1)}$};
    
    \draw[dashed,blue,->] (60,7.5) -- (7.5,5);
    
    \draw[dashed,blue,->] (60,7.5) -- (10.5,7.6);
    
    \draw[dashed,blue,->] (60,12.5) -- (15.25,12.5);
    
    \draw[dashed,blue,->] (60,12.5) -- (21.5,18.5);
    
    \draw[dashed,blue,->] (100,15) -- (32,29);
    
    \draw[dashed,blue,->] (100,15) -- (47,43);
    
    \draw[dashed,blue,->] (100,15) -- (69,65);
    
    \draw (120,80) rectangle (230,90) node[pos=.5] {$n_{_e}^{((\ell+1)_1)}$ h.e.p. between $n_{_p}^{(\ell+1)}$ and $n_{_p}^{(\ell+2)}$};
    
    \draw[dashed,blue,->] (165,90) -- (153.5,150);
    
    \draw[black,<->] (-5,0) -- (245,0);
    
    \draw[black,<->] (0,-5) -- (0,245);
    
    \draw[dashed,black,->] (0,0) -- (245,245);

    \node at (3,-6) {$1$};

    \node at (1,0) {|};
    
    \node at (20,-6) {$20$};

    \node at (20,0) {|};
    
    \node at (40,-6) {$40$};

    \node at (40,0) {|};
    
    \node at (60,-6) {$60$};

    \node at (60,0) {|};
    
    \node at (80,-6) {$80$};

    \node at (80,0) {|};
    
    \node at (100,-6) {$100$};

    \node at (100,0) {|};
    
    \node at (120,-6) {$120$};

    \node at (120,0) {|};
    
    \node at (140,-6) {$140$};

    \node at (140,0) {|};
    
    \node at (160,-6) {$160$};

    \node at (160,0) {|};
    
    \node at (180,-6) {$180$};

    \node at (180,0) {|};
    
    \node at (200,-6) {$200$};

    \node at (200,0) {|};
    
    \node at (220,-6) {$220$};   
    
    \node at (220,0) {|};
    
    \node at (238,-6) {$(\times 10^6)$};

    \node at (252,-1) {$n$};
    
    \node at (-6,3) {$1$};

    \node at (0,1) {{\bf--}};
        
    \node at (-6,20) {$20$};

    \node at (0,20) {{\bf--}};
    
    \node at (-6,40) {$40$};

    \node at (0,40) {{\bf--}};
    
    \node at (-6,60) {$60$};

    \node at (0,60) {{\bf--}};
    
    \node at (-6,80) {$80$};

    \node at (0,80) {{\bf--}};
    
    \node at (-7.5,100) {$100$};

    \node at (0,100) {{\bf--}};
    
    \node at (-7.5,120) {$120$};

    \node at (0,120) {{\bf--}};
    
    \node at (-7.5,140) {$140$};

    \node at (0,140) {{\bf--}};
    
    \node at (-7.5,160) {$160$};

    \node at (0,160) {{\bf--}};
    
    \node at (-7.5,180) {$180$};

    \node at (0,180) {{\bf--}};
    
    \node at (-7.5,200) {$200$};

    \node at (0,200) {{\bf--}};
    
    \node at (-7.5,220) {$220$};

    \node at (0,220) {{\bf--}};

    \node at (-10,235) {$(\times 10^6)$};

    \node at (-1,250) {$J_3(n)$};
    
    \node[circle,color=black, fill=red, inner sep=0pt,minimum size=4pt,label=below right:{}] (b) at (3.9,3.9) {};
    
    \node[circle,color=black, fill=blue, inner sep=0pt,minimum size=3pt,label=below right:{}] (b) at (5.9,4.9) {};
    
    \node[circle,color=black, fill=blue, inner sep=0pt,minimum size=3pt,label=below right:{}] (b) at (8.9,7.9) {};
    
    \node[circle,color=black, fill=blue, inner sep=0pt,minimum size=3pt,label=below right:{}] (b) at (13.5,12.5) {};
    
    \node[circle,color=black, fill=blue, inner sep=0pt,minimum size=3pt,label=below right:{}] (b) at (20.2,19.2) {};
    
    \node[circle,color=black, fill=blue, inner sep=0pt,minimum size=3pt,label=below right:{}] (b) at (30.4,29.4) {};
    
    \node[circle,color=black, fill=blue, inner sep=0pt,minimum size=3pt,label=below right:{}] (b) at (45.4,44.4) {};
    
    \node[circle,color=black, fill=blue, inner sep=0pt,minimum size=3pt,label=below right:{}] (b) at (68.1,67.1) {};
    
    \node[circle,color=black, fill=red, inner sep=0pt,minimum size=4pt,label=below right:{}] (b) at (102.2,102.2) {};
    
    \node[circle,color=black, fill=blue, inner sep=0pt,minimum size=3pt,label=below right:{}] (b) at (153.2,152.2) {};
    
    \node[circle,color=black, fill=red, inner sep=0pt,minimum size=4pt,label=below right:{}] (b) at (229,229) {};

    \node at (250,250) {$J_{_3}(n)=n$};
   
    \end{tikzpicture}}
    \caption{Fixed points $n_{_p}^{(\ell)}$, $n_{_p}^{(\ell+1)}$ and $n_{_p}^{(\ell+2)}$ with $\ell=17$ and pure high extremal points (h.e.p.) between them.}
    \label{fig-3-fixed}
\end{figure}

The above Figure \ref{fig-3-fixed} illustrates three consecutive fixed points  $n_{_p}^{(\ell)}=\num{3986218}$,  $n_{_p}^{(\ell+1)}=\num{102162424}$, and $n_{_p}^{(\ell+2)}=\num{229865455}$ and the pure high extremal points between them. The pure high extremal points are represented by blue dots and the fixed points by red dots. 

Our goal is to find a procedure to compute consecutive fixed points skipping the pure high extremal points in between. Next, we present a recurrence formula for the fixed points of $J_{_3}$ that uses the number $\overline{m}_{\ell}$, which is the main result of this section.

\begin{theorem}[Recurrence Formula for Fixed Points]\label{recurrence-formula}
Let $\{n_{_p}^{(\ell)}\}_{\ell\in\mathbb{N}}$ be the sequence of fixed points for $J_{_3}$, initialized with $n_{_p}^{(1)} = 1$. Then, the number of pure high extremal points between $n_{_p}^{(\ell)}$ and $n_{_p}^{(\ell+1)}$, $\overline{m}_{\ell}$, is given by the following formula:
\begin{equation}\label{m-l-formula}\overline{m}_{\ell} = \max\left\{m \in \mathbb{Z}_{+} \mid 2^m \text{ divides } 3n_{_p}^{(\ell)} + 2\right\} \text{ for each } \ell,\end{equation}
and the following recurrence formula holds for $\ell \in \mathbb{N}$: 
\begin{equation}\label{star-formula}
n_{_p}^{(\ell+1)} = \dfrac{3^{\overline{m}_{\ell}}(3n_{_p}^{(\ell)} + 2) - 2^{\overline{m}_{\ell}}}{2^{\overline{m}_{\ell} + 1}}.
\end{equation}
\end{theorem}

\begin{proof} For simplicity, we set $\overline{m}$ instead of $\overline{m}_{\ell}$ for every fixed $\ell\in \mathbb{N}$. 
    
First, when $\max\left\{m \in \mathbb{Z}_{+} \mid 2^m \text{ divides } 3n_{_p}^{(\ell)} + 2\right\}=0$, $3n_{_p}^{(\ell)} + 2$ is odd. This implies that $n_{_p}^{(\ell)}$ is also odd. Then, according to Theorem \ref{Thm-cases}(iv), the subsequent high extremal point is a fixed point, which means that there are not pure high extremal points in between, i.e.,  $\overline{m}=0$. Hence, equation \eqref{m-l-formula} holds and the next fixed point is given by: 
\[n_{_p}^{(\ell+1)} = \dfrac{3n_{_p}^{(\ell)} + 1}{2},\]
which aligns with the expression in \eqref{star-formula} for $\overline{m} = 0$.  
    
Next, if $\max\left\{m \in \mathbb{Z}_{+} \mid 2^m \text{ divides } 3n_{_p}^{(\ell)} + 2\right\}=1$, $3n_{_p}^{(\ell)} + 2$ is divisible only by $2$. Therefore, $n_{_p}^{(\ell)}$ is even. It follows from Theorem \ref{Thm-cases}(iii) that the next high extremal point is pure and satisfies:
\[n_{_e}^{(\ell_1)} = \dfrac{3n_{_p}^{(\ell)}}{2}.\] 
Since $3n_{_p}^{(\ell)} + 2$ is not divisible by $4$, $n_{_e}^{(\ell_1)}$ is even. Invoking Theorem \ref{Thm-cases}(i), the next high extremal point is actually a fixed point, which means that the number of pure high extremal points is $1$ ($\overline{m}=1$), and hence equation \eqref{m-l-formula} holds. Moreover, the next fixed point is given by: 
\[n_{_p}^{(\ell+1)} = n_{_e}^{(\ell_1 + 1)} = \dfrac{3n_{_e}^{(\ell_1)} + 2}{2} = \dfrac{3^1(3n_{_p}^{(\ell)} + 2) - 2^1}{2^2}.\] 
This confirms the validity of formula \eqref{star-formula} for $\overline{m}=1$.

Let $\check m:=\max\left\{m \in \mathbb{Z}_{+} \mid 2^m \text{ divides } 3n_{_p}^{(\ell)} + 2\right\}$. Assume that $\check m\ge 2$. So, it is possible to apply Theorem \ref{Thm-cases}(iii) once and Theorem \ref{Thm-cases}(ii) $\check m-1$ times, to get that there is an increasing sequence of pure high extremal points, $\{n_{_e}^{(\ell_i)}\}_{i=1}^{\check m}$, between $n_{_p}^{(\ell)}$ and $n_{_p}^{(\ell+1)}$. The last pure extremal point, $n_{_e}^{(\ell_{\check{m}})}$, is computed by:
\begin{equation}\label{ne-mbar}n_{_e}^{(\ell_{\check{m}})} = \dfrac{3^{\check{m}-1}(3n_{_p}^{(\ell)}+2)}{2^{\check{m}}} - 1.\end{equation}
We claim now that $n_{_e}^{(\ell_{\check{m}})}$ is even. A contrary assumption implies that $3^{\check{m}-1}(3n_{_p}^{(\ell)}+2)/2^{\check{m}}$ is divisible by $2$. This in turn yields that $3n_{_p}^{(\ell)}+2$ is divisible by a power of $2$ greater than $2^{\check{m}}$, which is a contradiction with the definition of $\check{m}$. Applying Theorem \ref{Thm-cases}(i), the next high extremal point is actually a fixed point, which proves that $\check{m}=\overline{m}$ and equation \eqref{m-l-formula} holds. Moreover, the next fixed point is given by:
\begin{equation}\label{last-eq-used}n_{_p}^{(\ell+1)} = n_{_e}^{(\ell_{\overline{m}}+1)} = \dfrac{3n_{_e}^{(\ell_{\overline{m}})} + 2}{2}.\end{equation}
However, examining our main recurrence formula, we deduce:
\begin{align*}
\dfrac{3^{\overline{m}}(3n_{_p}^{(\ell)}+2)-2^{\overline{m}}}{2^{\overline{m}+1}} &= \dfrac{\dfrac{3^{\overline{m}}(3n_{_p}^{(\ell)}+2)}{2^{\overline{m}}}-1}{2} \\
&= \dfrac{3\left(\dfrac{3^{\overline{m}-1}(3n_{_p}^{(\ell)}+2)}{2^{\overline{m}}}-1\right)+2}{2}\\
&=\dfrac{3n_{_e}^{(\ell_{\overline{m}})}+2}{2}\\&=n_{_p}^{(\ell+1)},
\end{align*} where we use \eqref{ne-mbar} in the third equality and \eqref{last-eq-used} in the last equality. This completes the proof of the theorem.
\end{proof}

Now we present a formula that allows us to solve the Josephus problem for a given $n$ using only fixed points.   
\begin{theorem}[Expression for Evaluating the Josephus Function]\label{Josephus-function-at-N}
Let \( n_{_p}^{(\ell)} \) be a fixed point of \( J_{_3} \) and let \(n\) be in the interval \([[n_{_p}^{(\ell)}+1, n_{_p}^{(\ell+1)}]]\). Define
\begin{equation}\label{I-ell}
\mathfrak{m}:=\left \lceil \log_{3/2}\left (\dfrac{2n+1}{3n_{_p}^{(\ell)}+2}\right) \right \rceil. 
\end{equation}
Then, we have
\begin{equation}\label{EQ1000}
J_{_3}(n)=3n+1-\left(\dfrac{2}{3} \right)^{\overline{m}_{\ell}-\mathfrak{m}}(2n_{_p}^{(\ell+1)}+1).
\end{equation}
\end{theorem}
\begin{proof} 
For clarity, let us use $\overline{m}$ instead of $\overline{m}_{\ell}$ for every fixed $\ell\in \mathbb{N}$. We divide the proof in two cases:

\noindent {\bf Case 1:} If $\overline{m}=0$, there are no pure high extremal points between $n_{_p}^{(\ell)}$ and $n_{_p}^{(\ell+1)}$ and formula \eqref{star-formula} of Theorem \ref{recurrence-formula} implies that
\(n_{_p}^{(\ell+1)}=(3n_{_p}^{(\ell)}+1)/2\) which combined with the fact that \[n_{_p}^{(\ell)}+\frac{1}{6}< n\le n_{_p}^{(\ell+1)}\] yields the inequality:
\[\left (\dfrac{3}{2}\right)^{-1}=\dfrac{2}{3}<\dfrac{2n+1}{3n_{_p}^{(\ell)}+2} \leq 1 =\left (\dfrac{3}{2}\right)^{0}.\]
Taking the logarithm base $3/2$ on each side of the last inequality and using the definition of $\mathfrak{m}$, we get that $\mathfrak{m}=0$. Now using equation (16) from Corollary 7 of \cite{Bel}, we find $J_{_3}(n)=3n-2n_{_p}^{(\ell+1)}$, which coincides with equation \eqref{EQ1000} when $\overline{m}=\mathfrak{m}=0$. This validates \eqref{EQ1000} for this case.

\noindent {\bf Case 2:} Assume that $\overline{m}\ge 1$. According to Theorem \ref{recurrence-formula}, there are exactly $\overline{m}$ pure high extremal points denoted $n_{_e}^{(\ell_1)}, \ldots, n_{_e}^{(\ell_{\overline{m}})}$ between $n_{_p}^{(\ell)}$ and $n_{_p}^{(\ell+1)}$. Also,  denote ${\rm I}_{1}:=[[n_{_p}^{(\ell)}+1,n_{_e}^{(\ell_1)}]]$, ${\rm I}_{j}:=[[n_{_e}^{(\ell_{j-1})}+1,n_{_e}^{(\ell_j)}]]$ for $j=2,\dots, \overline{m}$, and ${\rm I}_{{\overline{m}+1}}:=[[n_{_e}^{(\ell_{\overline{m}})}+1,n_{_p}^{(\ell+1)}]]$. Additionally, Theorem \ref{Thm-cases}(iii) implies
\begin{equation}\label{n-e-ell-1} n_{_e}^{(\ell_1)}=3n_{_p}^{(\ell)}/2.\end{equation} Moreover, it is possible to apply Theorem \ref{Thm-cases}(iii) once and Theorem \ref{Thm-cases}(ii) $j-1$ times, to get that 
\begin{equation}\label{n-e-ell-j}
n_{_e}^{(\ell_{j})}=\dfrac{3^{j-1}(3n_{_p}^{(\ell)}+2)}{2^{j}}-1
\end{equation} for $j=2,\ldots,\overline{m}$, 
and Theorem \ref{recurrence-formula} for $\overline{m}$ yields \begin{equation}\label{n-p-ell-1} n_{_p}^{(\ell+1)}=\dfrac{3^{\overline{m}}(3n_{_p}^{(\ell)} + 2) - 2^{\overline{m}}}{2^{\overline{m} + 1}}.\end{equation}

Next, we analyze three subcases:

{\bf Subcase 2a:} $n\in {\rm I}_{1}:=[[n_{_p}^{(\ell)}+1,n_{_e}^{(\ell_1)}]]$. Using a similar argument as {\bf Case 1}, we obtain the following inequality:
$$\left (\dfrac{3}{2}\right)^{-1}=\dfrac{2}{3}<\dfrac{2n+1}{3n_{_p}^{(\ell)}+2} \leq 1 =\left (\dfrac{3}{2}\right)^{0}.$$ This inequality together with the definition of $\mathfrak{m}$ imply that $\mathfrak{m}=0$. It follows from  equation (16) in Corollary 7 of \cite{Bel} and the fact that $n_{_e}^{(\ell_1)}$ is a pure extremal point, i.e., $J_{_3}(n_{_e}^{(\ell_1)})= n_{_e}^{(\ell_1)}-1$ that
\begin{align*}
J_{_3}(n)&=3(n-n_{_e}^{(\ell_1)})+J_{_3}(n_{_e}^{(\ell_1)}) \\
                      &=3n-2n_{_e}^{(\ell_1)}-1 \\
                      &=3n-3n_{_p}^{(\ell)}-1 \\
                      &=3n-3\left(\frac{1}{3}\left(\dfrac{2^{\overline{m}+1}n_{_p}^{\ell+1}+2^{\overline{m}}}{3^{\overline{m}}}-2\right)\right)-1 \\
                      &=3n+1-\left(\dfrac{2}{3}\right)^{\overline{m}}(2n_{_p}^{(\ell+1)}+1),
\end{align*} where we have used \eqref{n-e-ell-1} in the third equality and \eqref{n-p-ell-1} in the fourth equality. Therefore, the validity of equation \eqref{EQ1000} is proved.

{\bf Subcase 2b:} $n\in {\rm I}_{j}=[[n_{_e}^{(\ell_{j-1})}+1,n_{_e}^{(\ell_j)}]]$ for $j=2,\dots, \overline{m}$. In this case, we use \eqref{n-e-ell-j} and the same formula at $j-1$ for any $j-1=2,\ldots, \overline{m}-1$ (when $j-1=1$, we use \eqref{n-e-ell-1}) together with the fact that $n_{_e}^{(\ell_{j-1})}<n\le n_{_e}^{(\ell_j)}$ for $j=2,\ldots,\overline{m}$, to prove that
\[\left (\dfrac{3}{2}\right)^{j-2}<\dfrac{2n+1}{3n_{_p}^{(\ell)}+2} \leq \left (\dfrac{3}{2}\right)^{j-1}.\]

Taking the logarithm base $3/2$ of the last inequality and using the definition of $\mathfrak{m}$, we get that $\mathfrak{m}=j-1$. Using now equation (16) in Corollary 7 of \cite{Bel} and the fact that $n_{_e}^{(\ell_j)}$ is a pure high extremal point, we obtain 
\begin{align*}
J_{_3}(n)&=3(n-n_{_e}^{(\ell_{j})})+J_{_3}(n_{_e}^{(\ell_{j})}) \\
                      &=3n-2n_{_e}^{(\ell_{j})}-1 \\
                      &=3n-2\left( \dfrac{3^{j-1}(3n_{_p}^{(\ell)}+2)}{2^{j}}-1\right)-1 \\
                      &=3n+1-\left(\dfrac{3}{2}\right)^{\mathfrak{m}}(3n_{_p}^{(\ell)}+2) \\
                      &=3n+1-\left(\dfrac{3}{2}\right)^{\mathfrak{m}}\left (\dfrac{2^{\overline{m}+1}n_{_p}^{(\ell+1)}+2^{\overline{m}}}{3^{\overline{m}}}\right )\\
                      &=3n+1-\left(\dfrac{3}{2}\right)^{\mathfrak{m}}\left(\dfrac{2}{3}\right)^{\overline{m}}(2n_{_p}^{(\ell+1)}+1)\\
                      &=3n+1-\left(\dfrac{2}{3}\right)^{\overline{m}-\mathfrak{m}}(2n_{_p}^{(\ell+1)}+1),
\end{align*}
where we have used \eqref{n-e-ell-j} in the third equality and \eqref{n-p-ell-1} in the fifth equality. Thus, equation \eqref{EQ1000} holds.

{\bf Subcase 2c:} $n \in {\rm I}_{\overline{m}+1}:=[[n_{_e}^{(\ell_{\overline{m}})}+1,n_{_p}^{(\ell+1)}]]$. Similarly using now that $n_{_e}^{(\ell_{\overline{m}})} = (3^{\overline{m}-1}(3n_{_p}^{(\ell)}+2))/2^{\overline{m}} - 1$ and \eqref{star-formula} from Theorem \ref{recurrence-formula} to prove  $$\left (\dfrac{3}{2}\right)^{\overline{m}-1}<\dfrac{2n+1}{3n_{_p}^{(\ell)}+2} \leq \left (\dfrac{3}{2}\right)^{\overline{m}}.$$ Taking the logarithm base $3/2$ on each side of the last inequality and using the definition of $\mathfrak{m}$, we have that $\mathfrak{m}=\overline{m}$. Equation \eqref{EQ1000} holds similarly to {\bf Case 1} since we again can use equation (16) from Corollary 7 of \cite{Bel} to write $J_{_3}(n)=3n-2n_{_p}^{(\ell+1)}$, which coincides with equation \eqref{EQ1000} when $\mathfrak{m}=\overline{m}$. 
\end{proof}

Note that the last two theorems are very powerful. Next, we present a direct application of those results for evaluating the Josephus function at a large value of $n$ and a new expression for computing the number of pure high extremal points between fixed points.

\begin{remark}[Evaluating $J_{_3}(\num{50000000})$]
Given $n_{_{p}}^{(17)}=\num{3986218}$ taken from Table \ref{fixed-points-values}, we can effortlessly compute the next fixed point $n_{_{p}}^{(18)}$ using the recurrence formula \eqref{star-formula}:
$$n_{_{p}}^{(18)}=\dfrac{3^{\overline{m}_{17}}(3n_{_{p}}^{(17)}+2)-2^{\overline{m}_{17}}}{2^{\overline{m}_{17}+1}}=\dfrac{3^{7}(3(\num{3986218})+2)-2^{7}}{2^{8}}=\num{102162424},$$
where $\overline{m}_{17}=7$. Moreover, we can evaluate $J_{_3}$ by hand at $n=\num{50000000}$ as
$$J_{_3}(\num{50000000})= 3(\num{50000000})+1-\left(\dfrac{2}{3} \right)^{\overline{m}_{17}-\mathfrak{m}}(2(\num{102162424})+1)=\num{13783435},$$
after observing that $$\mathfrak{m}=\left \lceil \log_{3/2}\left (\dfrac{2n+1}{3n_{_p}^{(17)}+2}\right) \right\rceil=\left\lceil\log_{3/2}\left(\dfrac{2(\num{50000000})+1}{3(\num{3986218})+2}\right)\right\rceil=\lceil5.2377252342894725\rceil=6.$$
\end{remark}

\begin{remark}[A New Formula for Computing $\overline{m}_{\ell}$]
We can use the formula for $J_{_3}(n)$ in terms of the fixed points of $J_{_3}$ provided in \eqref{EQ1000} of Theorem \ref{Josephus-function-at-N} to compute 
$$J_{_3}(n_{_p}^{(\ell+1)})=3n_{_p}^{(\ell+1)}+1-\left(\dfrac{2}{3} \right)^{\overline{m}_{\ell}-\mathfrak{m}}(2n_{_p}^{(\ell+1)}+1)=n_{_p}^{(\ell+1)},$$
which implies
$$\left (1-\left (\dfrac{2}{3} \right )^{\overline{m}_{\ell}-\mathfrak{m}} \right)(2n_{_p}^{(\ell +1)}+1 )=0.$$ Hence, $\overline{m}_{\ell}=\mathfrak{m}$ and
$${\overline{m}_{\ell}}=\left \lceil \log_{3/2}\left (\dfrac{2n_{_p}^{(\ell+1)}+1}{3n_{_p}^{(\ell)}+2}\right) \right \rceil.$$
\end{remark}

\subsection{A Fixed Point Algorithm for Solving the Josephus Problem}

To solve the Josephus problem, we propose the {\it fixed point} algorithm. This strategy computes recursively the fixed points $n^{(\ell)}_{_p}$ for $\ell=1,2,\ldots,q$ until $n^{(q)}_{_p}$ is greater than or equal to $n$. This is detailed in Theorem \ref{recurrence-formula}. Subsequently, $J_{_3}$ is evaluated at $n$ as given by Theorem \ref{Josephus-function-at-N}. The iteration starts with $n_{_p}^{(1)}=1$ and from $\ell=1,2,\ldots$ until $n_{_e}^{(\ell+1)}\geq n$, we compute:
\begin{equation*}
n_{_{p}}^{(\ell+1)} = \dfrac{3^{\overline{m}_{\ell}}(3n_{_{p}}^{(\ell)}+2)-2^{\overline{m}_{\ell}}}{2^{\overline{m}_{\ell}+1}},
\end{equation*}
where $\overline{m}_{\ell}= \max\left\{m \in \mathbb{Z}_{+} \mid 2^m \text{ divides } 3n_{_p}^{(\ell)} + 2\right\}$. Then, we evaluate
$$J_{_3}(n)=3n+1-\left(\dfrac{2}{3} \right)^{\overline{m}_{\ell}-\mathfrak{m}}(2n_{_p}^{(\ell+1)}+1),$$
where $\mathfrak{m}=\left \lceil \log_{3/2}\left (\dfrac{2n+1}{3n_{_p}^{(\ell)}+2}\right) \right \rceil$.

Comparing the fixed point algorithm with the extremal algorithm introduced in \cite{Bel}, our collected data suggests that the former is approximately \(50\%\) more efficient than the latter. Notably, for even fixed points, the number of pure high extremal points, \(\overline{m}_{\ell}\), tends to \(1\) as \(n\) becomes sufficiently large. This indicates that the number of extremal points is approximately the double of the number of fixed points. As a result, the extremal algorithm requires roughly twice the number of iterations as the fixed point algorithm. For example, to compute \(J_{_3}(\num{50000000})\), the fixed point algorithm requires the computation of the first \(18\) fixed points. On the other hand, the extremal algorithm also needs to compute \(\sum_{\ell=1}^{17}\overline{m}_{\ell}=27\) pure extremal points, leading to a total of \(45\) high extremal points (including the fixed points). In this scenario, the fixed point algorithm is approximately \(52\%\) faster than the extremal algorithm. However, as \(n\) increases, this percentage will approach \(50\%\). Next, we graph the function 
\[
r(q)=\sum_{\ell=1}^{q-1} \left(1-\frac{q}{(\overline{m}_{\ell}+q)}\right) \times 100\%,
\] which represents the percent of gain of the fixed point algorithm over the extremal algorithm for the values \(q=1,2,\ldots,39\) from Table \ref{fixed-points-values} and \(q=1,2,\ldots,200\).

\begin{figure}[h]
    \centering
    \begin{subfigure}[b]{0.495\textwidth}
        \centering
        \includegraphics[width=\textwidth]{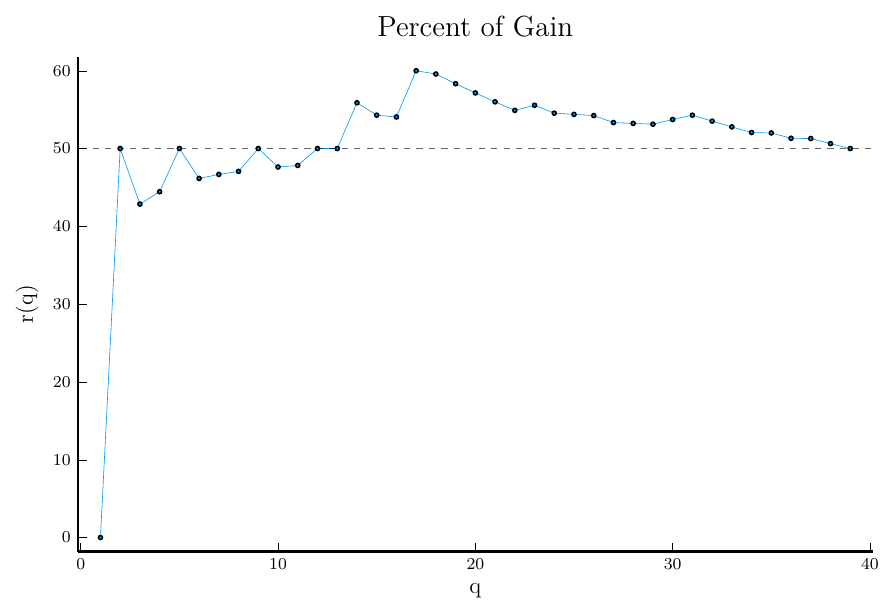}
        \caption{Graph of $r(q)$ the percentage gain for \(q=1,2,\ldots,39\).}
        \label{f3_sub}
    \end{subfigure}
    \hfill
    \begin{subfigure}[b]{0.495\textwidth}
        \centering
        \includegraphics[width=\textwidth]{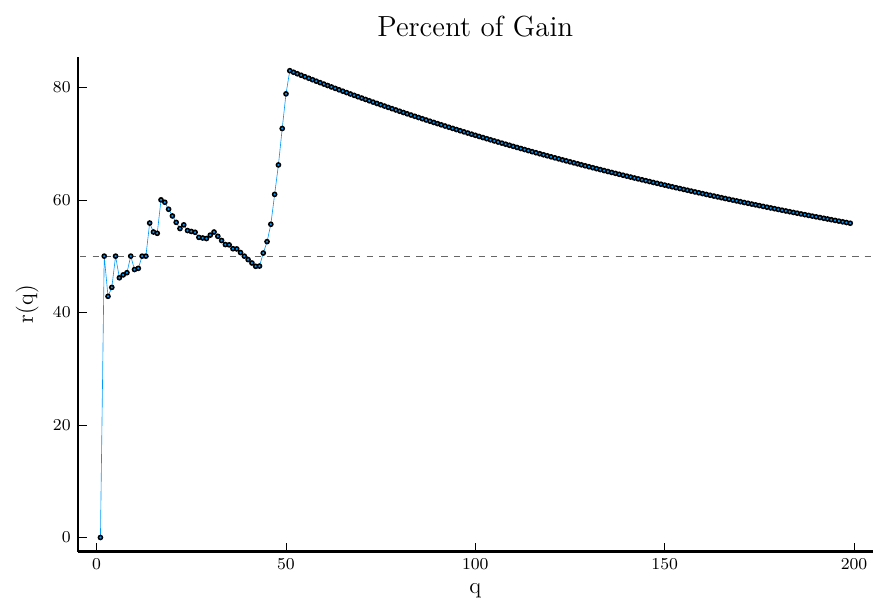}
        \caption{Graph of $r(q)$ the percentage gain for \(q=1,2,\ldots,200\).}
        \label{f4_sub}
    \end{subfigure}
    \caption{Graphs of $r(q)$ showing the percentage gain of the fixed point algorithm.}
    \label{fig:combined_figure}
\end{figure}




\section{Concluding Remarks}\label{sec:conclusion}
In this paper, we derived a recurrence formula for computing successive fixed points of the Josephus function, specifically when the reduction constant is three. The proposed recurrence relation not only partially solves the open question introduced in \cite{Bel} but also enhances the efficiency of the non-recursive algorithm previously presented in \cite{Bel} for large inputs.
Our results pave the way for further research in this domain. An immediate direction is to generalize the recurrence formula for different reduction constants. Finding a recurrence formula for the general case remains an open problem, and its solution could provide a deeper understanding and broaden the applications of this problem.

\begin{acknowledgment}{Acknowledgment.}
The first author would like to express gratitude for the support provided by the NSF grant DMS -- 2307328 and by the internal research and artistry (R\&A) grant at Northern Illinois University. 
\end{acknowledgment}

\begin{biog}
\item[Yunier Bello-Cruz] earned his Bachelor and Master degrees in mathematics from the University of Havana ($2002$ and $2005$, Havana) and his Ph.D. in Mathematics from the Instituto de Matematica Pura e Aplicada (IMPA) ($2009$, Rio de Janeiro). He has held various academic positions, including a Postdoctoral Research Fellowship at the University of British Columbia ($2013$--$2015$, Vancouver) and tenured Assistant Professor roles at Federal University of Goias ($2009$--$2016$, Goiania). Currently, Bello-Cruz is an Associate Professor at Northern Illinois University, where his research focuses on numerical analysis and continuous optimization. Bello-Cruz has received numerous research grants, including CNPq and CAPES Grants from Brazil, and NSF Grants from US, which have supported his work.

\begin{affil}
Department of Mathematical Sciences, Northern Illinois University, DeKalb IL 60115\\
yunierbello@niu.edu
\end{affil}

\item[Roy Quintero--Contreras] received his Bachelor, Master, and Ph.D. degrees in mathematics from
Universidad Central de Venezuela ($1985$, Caracas), Universidad de Los Andes ($1992$, Merida),
and the University of Iowa (1997, Iowa City), respectively. He currently holds an Instructor
position at the Northern Illinois University. His previous appointments were at the University
of Southern California as a Lecturer ($2019$--$2021$) and at the University of Iowa as a visiting
professor ($2016$--$2019$). He retired from Universidad de Los Andes, Trujillo-Venezuela, as a
Full Professor where he worked for more than $25$ years.
\begin{affil}
Department of Mathematical Sciences, Northern Illinois University, DeKalb IL 60115\\
rquinterocontreras@niu.edu
\end{affil}
\end{biog}
\vfill\eject

\end{document}